\documentclass[10pt,a4paper,twoside]{amsart}

\usepackage{amsfonts, amssymb, amsmath, amsthm}
\usepackage{mathrsfs} % \mathscr
\usepackage{latexsym}%utilisation des symboles LaTeX pour avoir un beau LaTeX
\usepackage{enumerate}
\usepackage{multicol}
\usepackage{verbatim}
\usepackage{hyperref}
\usepackage{url}
\usepackage{csquotes}
\usepackage{mathabx} %% for widecheck
\newtheorem{thm}{Theorem}
\newtheorem{lem}[thm]{Lemma}

\theoremstyle{remark}

\let\ve=\varepsilon

\DeclareMathOperator{\Log}{log}

\def\Oc{\mathcal{O}}
\def\Ocal{\mathcal{O}}

\def\1{1\!\!\!1}

\let\ve=\varepsilon

\title{An explicit upper bound for $L(1,\chi)$ when $\chi$ is quadratic} 
\author{D. R. Johnston} 
\address{School of Science\\
UNSW Canberra at ADFA\\
ACT, Australia}
\email{daniel.johnston@adfa.edu.au}

\author{O. Ramar\'e}
\address{CNRS, Aix Marseille Universit\'e, I2M, Marseille, France }
\email{olivier.ramare@univ-amu.fr}

\author{T. Trudgian}
\address{School of Science\\
UNSW Canberra at ADFA\\
ACT, Australia}

\email{t.trudgian@adfa.edu.au}


\begin{document}
\hypersetup{pageanchor=false}
\maketitle
\hypersetup{pageanchor=true}

\begin{abstract}
  %\texttt{File \jobname.tex} 
  \noindent
  We consider Dirichlet $L$-functions $L(s, \chi)$ where $\chi$ is a
  non-principal quadratic character to the modulus $q$. We make
  explicit a result due to Pintz and Stephens by showing that
  $|L(1, \chi)|\leq \frac{1}{2}\log q$ for all $q\geq 2\cdot 10^{23}$
  and $|L(1, \chi)|\leq \frac{9}{20}\log q$ for all $q\geq 5\cdot 10^{50}$.
\end{abstract}

%%%%%%%%%%%%%%%%%%%%%%%%%%%%%%%%%%%%%%%%%%%%%%%%%%%%%%%%%
\section{Introduction and results}\label{cello}
%%%%%%%%%%%%%%%%%%%%%%%%%%%%%%%%%%%%%%%%%%%%%%%%%%%%%%%%%

A central problem in number theory concerns estimates on $L(1, \chi)$,
where $\chi$ is a non-principal Dirichlet character to the modulus
$q$, and where $L(s, \chi)$ is its associated Dirichlet
$L$-function. Bounding sums of $\chi(n)$ trivially leads to the bound
$|L(1, \chi)| \leq \log q + O(1)$. The P\'{o}lya--Vinogradov
inequality allows one to improve this to $(1/2)\log q + O(1)$. An
interesting history of these developments is given by Pintz
\cite{PintzVII}.

Explicit versions of the above results date back to Hua
\cite{Hua}. See also work by Louboutin \cite{Louboutin*01} and the second
author \cite{Ramare*01,Ramare*02c} for finding small pairs $c, q_{0}$ such that
$L(1, \chi)\leq (1/2) \log q + c$ for all $q\geq q_{0}$. It appears
difficult to improve on these bounds for generic $q$.

When $q$ is prime, the best result is due to Stephens \cite{Stephens},
namely that $|L(1, \chi)|\leq \frac{1}{2}(1 - e^{-1/2} + o(1))\log q$,
where $\frac{1}{2}(1- e^{-1/2}) = 0.1967\ldots$. This result has been
extended to arbitrary moduli by Pintz in~\cite{Pintz*77-5, PintzVII}. We aim at making
the Pintz--Stephens result partially explicit in the following theorems.

%%%%%%%%%%%%%%%%%%%%%%%%%
\begin{thm}\label{goal}
Let $\chi$ be a quadratic odd primitive  Dirichlet character
  modulo $q\ge 2\cdot 10^{23}$. We have $L(1,\chi)\le (\log q)/2$. 
\end{thm}
%%%%%%%%%%%%%%%%%%%%%%% 
For even characters, this is proved for $q\ge2$ in \cite{Ramare*01}
after several papers by Louboutin, the last of which is \cite{Louboutin*01}.
Bounds relying on additional constraints on the characters at the
small primes have been investigated by Louboutin in
\cite{Louboutin*02b}, by the second author in \cite{Ramare*02c}, by Saad Eddin in
\cite{Saad-Eddin*16} and by Platt and Saad Eddin in \cite{Platt-SaadEddin*13}.
On taking $q$ to be larger, we can improve on the factor $1/2$ in
Theorem \ref{goal}.
%%%%%%%%%%%%%%%%%%%%%%%%
\begin{thm}\label{goal2}
Let $\chi$ be a quadratic primitive Dirichlet character modulo $q$. 
The inequality $L(1,\chi)\le (9/40)\log q$ holds true when $\chi$ is
 even and $q\ge 2\cdot 10^{49}$ or $\chi$ is odd and $q\geq 5\cdot 10^{50}$.
\end{thm}
%%%%%%%%%%%%%%%%%%%%%%%
%We have $\tfrac9{20}=0.45$.

We note that, on the Generalized Riemann hypothesis much more is
known. Littlewood \cite{Littlewood} showed that
$L(1, \chi) \ll \log\log q$. This has been made explicit for large $q$
in \cite{LLS} by Lamzouri, Li and Soundararajan,
and then for all $q$ in \cite{Languasco} by Languasco and the third author. Finally,
although we do not consider lower bounds on $L(1, \chi)$, we direct
the reader to a survey of explicit and inexplicit bounds of
Mossinghoff, Starichkova and the third author in
\cite{MST}, and to the recent work \cite{Langsolo}.

The outline of this paper is follows. In \S \ref{captain2} we collect
the necessary explicit results on character sums. In \S \ref{captain3}
we prepare the technical preliminaries to Stephens' approach, and
analyse these in \S \ref{captain4}. Our \S \ref{captain5} is purely
centred on the optimization in (an improved version of) Stephens' method, and contains no
number-theoretic input. Finally, in \S \ref{captain6} we prove
Theorems \ref{goal} and \ref{goal2}.

%%%%%%%%%%%%%%%%%%
We use the notation $f(x)=\mathcal{O}^*(g(x))$ to mean that
$f(x)\leq |g(x)|$ for the range of $x$ considered. We also make use of
the following notation.
%%%%%%%%%%%%%%%%%
We define
\begin{equation}
  \label{defh}
  h(\chi,y)=\sum_{n\le H^y}\frac{\chi(n)\Lambda(n)}{n},\quad
  h(1,y)=\sum_{n\le H^y}\frac{\Lambda(n)}{n}
\end{equation}
as well as
%We define $f(x)$ and $F(x)$ as partial sums and integrals:
\begin{equation}
  \label{deff}
  f(x)=\sum_{n\le H^x}\frac{\chi(n)}{H^x}, \quad 
  F(x)=\int_0^x f(t)dt
  =\sum_{n\le H^x}\frac{\chi(n)}{n\log H}
  -\frac{f(x)}{\log H}.
\end{equation}
 Our aim is to
 majorize~$F(1)$.
We further define
\begin{equation}
  \label{defell}
  \ell(y)=H^{-y}\sum_{n\le H^y}\chi(n)\log n.
\end{equation}
It is also convenient to introduce the points
\begin{equation}
  \label{defxm}
  x_m=1-\frac{\log m}{\log H}.
\end{equation}

%%%%%%%%%%%%%%%%%%%%%%%%%%%%%%%%%%%%%%%%%%%%%%%%%%%
\section{Preliminary results}\label{captain2}
%%%%%%%%%%%%%%%%%%%%%%%%%%%%%%%%%%%%%%%%%%%%%%%%%%%
 
We now list a trivial result that follows immediately from partial summation.
%%%%%%%%%%%%%%%%%%%
\begin{lem}
 \label{loc1}
 When $x\ge0$, we have $\sum_{n\le x}1/\sqrt{n}\le 2\sqrt{x}$.
\end{lem}
%%%%%%%%%%%%%%%%%%%
The following result is slightly more subtle. %%%%%%%%%%%%%%%%%%%
\begin{lem}
 \label{loc2}
 When $x\ge y\ge1$, we have $\sum_{y\le n\le x}1/n\le 1+\log(x/y)$.
\end{lem}
%%%%%%%%%%%%%%%%%%%
\begin{proof}
Using Euler--Maclaurin summation one can show that
$$\sum_{n\leq x} \frac{1}{n} = \log x + \gamma + \frac{\left(\frac{1}{2} - \{x\}\right)}{x} + \mathcal{O}^*\left(\frac{1}{8x^{2}}\right),$$
whence
\begin{equation}\label{net}
\sum_{y\leq n\leq x} \frac{1}{n} = \log(x/y) + \mathcal{O}^*\left( \frac{1}{2x} + \frac{1}{2y} + \frac{1}{8x^{2}} + \frac{1}{8y^{2}}\right).
\end{equation}
The lemma is clearly true when $x=1$. Therefore, for $x\geq 2$ and
$y\geq 1$ we have, by \eqref{net} that
$\sum_{y\leq n \leq x} n^{-1} -\log(x/y)=\Ocal^*(29/32)$, and we are done.
\end{proof}

%%%%%%%%%%%%%%%%%%%%%%%%%%%%%%%%%%%%%%%%%%%%%%%%%%%
%\section{Bounds on primes}
%%%%%%%%%%%%%%%%%%%%%%%%%%%%%%%%%%%%%%%%%%%%%%%%%%%
We now list some bounds related to the prime number theorem. The first
is (a simplification of) a classical result from Rosser and
Schoenfeld, see \cite[Thm 12]{RS62}.
%%%%%%%%%%%%%%%%%%%
\begin{lem}
 \label{appLambda0}
 When $x\ge 0$, we have $\psi(x)\le 1.04\,x$.
\end{lem}
We note that the result of Rosser and Schoenfeld gives $1.03883$ in
Lemma \ref{appLambda0}, which is an approximation to
$\psi(113)/113$. To improve the bound in Lemma \ref{appLambda0} it
would be necessary to take $x\geq x_{0} > 113$, which, while possible,
would complicate greatly the ensuing analysis for only a marginal
improvement.
%%%%%%%%%%%%%%%%%%%

The second is an explicit bound of the form $\psi(x) -x = o(x)$ coming
from \cite[Table 15]{BKLNW} by Broadbent, Kadiri, Lumley, Ng, and Wilk.
%
%%%%%%%%%%%%%%%%%%%
\begin{lem}
 \label{appLambda}
 When $x\geq 10^5$, we have $|\psi(x)-x|\le 0.64673\,x/(\log x)^2$.
\end{lem}
%%%%%%%%%%%%%%%%%%%
%
On the Riemann hypothesis we have $\psi(x) -x = O(x^{1/2 +
  \epsilon})$. The following result, from \cite[Thm 2]{Buethe} of
B\"uthe,
gives an explicit version of an even sharper bound for a finite range. 
%%%%%%%%%%%%%%%%%%%
\begin{lem}
 \label{appLambda2}
 When $11< x\le 10^{19}$, we have $|\psi(x)-x|\le 0.94\sqrt{x}$.
\end{lem}
%%%%%%%%%%%%%%%%%%%
We remark that slightly weaker versions of Lemma \ref{appLambda2}, but
ones that hold in a longer range of $x$ have been provided by the
first author in \cite{Johnston}.
We require the following result to be used in tandem with Lemma \ref{appLambda2}.
%%%%%%%%%%%%%%%%%%%
\begin{lem}
 \label{appLambda3}
 When $e^{40}\le x$, we have $|\psi(x)-x|\le 1.994\cdot10^{-8}\,x$.
\end{lem}
%%%%%%%%%%%%%%%%%%%
This is obtained directly from \cite[Table 8]{BKLNW}. The key feature
here is that $e^{40}<10^{19}$ so that Lemma \ref{appLambda2} and Lemma
\ref{appLambda3} between them cover all values of $x> 11$. Better
results are available when $x$ is very large, say $\log x \geq 1000$
--- see \cite{PT} by Platt and Trudgian, and \cite{JY} by the first
author and Yang --- but Lemmas \ref{appLambda2} and
\ref{appLambda3} suffice for our needs.

We now turn to estimates on $\tilde\psi(u):= \sum_{n\leq u}\Lambda(n)/n$ to aid in the evaluation of $h(\chi, y)$ and $h(1,y)$ in (\ref{defh}). To obtain such estimates we correct a result of the second author in \cite{Ramare2013}.

%%%%%%%%%%%%%%%%%%%%%%%%%%%%%%%%%%
\begin{lem}\label{Ramcorrected}
    For $x\geq 71$ we have
    \begin{align*}
        \sum_{n\leq x}\frac{\Lambda(n)}{n}=\log x-\gamma+\frac{\psi(x)-x}{x}+\frac{0.047}{\sqrt{x}}+\frac{\log(2\pi)+10^{-4}}{x}+E(x),
    \end{align*}
    where
    \begin{equation*}
        E(x)=
        \begin{cases}
            1.75\cdot 10^{-12},&1\leq x< 2R\log^2 T_0\\
            \frac{1+2\sqrt{(\log x)/R}}{2\pi}\exp(-2\sqrt{(\log x)/R}), &x\geq 2R\log^2 T_0,
        \end{cases}
    \end{equation*}
    with\footnote{The value of $R$ comes from work by Kadiri
      \cite{Kadiri} on the classical zero-free region for the
      zeta-function. This can be lowered using more recent results
      \cite{MT} and \cite{MTY}, respectively by Mossinghoff and
      Trudgian and by Mossinghoff, Trudgian and Yang, but it is
      inconsequential for our purposes.} $R=5.69693$ and
    $T_0=2.44\cdot 10^{12}$.
\end{lem}
\begin{proof}
  As discussed by Chirre, Simoni\v{c}, and Val\aa s Hagenin in
  \cite{CSV}, by fixing a couple of small typos, Lemma 2.2 in
  \cite{Ramare2013} can be replaced by
    \begin{align*}
        &\sum_{n\leq x}\frac{\Lambda(n)}{n}=\log x-\gamma+\frac{\psi(x)-x}{x}-\sum_{\rho}\frac{x^{\rho-1}}{\rho(\rho-1)}\\
        &\qquad\qquad\qquad\qquad\qquad\qquad\qquad\qquad+\int_x^\infty\frac{\log 2\pi+\frac{1}{2}\log(1-u^{-2})}{u^2}\mathrm{d}u.
    \end{align*}
    Following \cite[\S 5]{Ramare2013}, we have
    \begin{equation*}
        \left|\sum_{\rho}\frac{x^{\rho-1}}{\rho(\rho-1)}\right|\leq \frac{0.047}{\sqrt{x}}+E(x).
    \end{equation*}
    Finally, since $x\geq 71$,
    \begin{align*}
        \left|\int_x^\infty\frac{\log 2\pi+\frac{1}{2}\log(1-u^{-2})}{u^2}\mathrm{d}u\right|&\leq\frac{\log(2\pi)}{x}+\frac{|\log(1-71^{-2})|}{2x}\\
        &\leq\frac{\log(2\pi)+10^{-4}}{x}.\qedhere
    \end{align*}
\end{proof}
\begin{lem}
  \label{psitildeasymp}
  We have
  \begin{align}
  %  \tilde\psi(X)=
    \sum_{n\le
      x}\Lambda(n)/n
    &=\Log x-\gamma+\Oc^*(1.3/\log^2x), \qquad x> 1,\label{lambdalog2}\\
    \sum_{n\le
      x}\Lambda(n)/n
    &=\Log x-\gamma+\Oc^*(1/\sqrt{x}), \qquad 1\leq x\leq 10^{19}.\label{lambdasqrt}
  \end{align}
\end{lem}
\begin{proof}
    Using Lemma \ref{Ramcorrected} with the bounds from Lemmas \ref{appLambda} and \ref{appLambda2}, we obtain that,
    \begin{equation*}
        \sum_{n\leq x}\frac{\Lambda(n)}{n}=\log x-\gamma+O^*\left(\frac{0.67}{\log^2 x}\right)
    \end{equation*}
    for $x\geq 10^5$, and
    \begin{equation*}
        \sum_{n\leq x}\frac{\Lambda(n)}{n}=\log x-\gamma+O^*\left(\frac{1}{\sqrt{x}}\right)
    \end{equation*}
    for $10^5\leq x\leq 10^{19}$. We then extend these estimates to smaller values of $x$ by direct computation, giving \eqref{lambdalog2} and \eqref{lambdasqrt}.
\end{proof}
%%%%%%%%%%%%%%%%%%%%%%%%%%%%%%%%%%

An immediate consequence of this result is as follows.
%%%%%%%%%%%%%%%%%%%%%
\begin{lem}
  \label{appLambdatilde}
  We have
  \begin{align*}
    \sum_{n\le
      x}\Lambda(n)/n
    \le\Log x-0.545, \qquad x\geq 10^3,\\
    \sum_{n\le
      x}\Lambda(n)/n
    \ge\Log x-0.576, \qquad x\geq 10^{6}.
  \end{align*}
\end{lem}
%%%%%%%%%%%%%%%%%%%%

%%%%%%%%%%%%%%%%%%%%%%
%%%%%%%%%%%%%%%%%%%%%%
We now examine the weighted average of $\sum_{n\leq u} \Lambda(n)/n$.
%%%%%%%%%%%%%%%%%%%%%
\begin{lem}
  \label{appLambdatilde1}
  We have
  \begin{equation*}
    \int_1^\infty\biggl|\sum_{n\le u}\frac{\Lambda(n)}{n}-\log u+\gamma\biggr|\frac{du}{u}
    \le
    0.411.
  \end{equation*}
\end{lem}
%%%%%%%%%%%%%%%%%%%%
This integral may be of interest in its own right. While the true
value of this integral seems close to $0.41$, we have no idea of the
conjectured limiting value of the integral. To this end, see a similar problem discussed in
\cite{Brent}.
%%%%%%%%%%%%%%%%%%%%%%
\begin{proof}
  We define $\Delta(u)=\sum_{n\le u}\Lambda(n)/n-\log
    u+\gamma$.
  When the variable $u$ is small, we compute directly by using the
  fact that $\tilde\psi(u)$ is constant on $[n,n+1)$ and that, with
  $\tau=\tilde\psi(n)+\gamma$, the integral 
  $\int_n^{n+1}|\Delta(u)|{du}/{u}$ is equal to
  \begin{equation*}
    \begin{cases}
      \frac{\log^2(n+1)-\log^2n}{2}-\tau\log\frac{n+1}{n}
      &\text{when $\tau\le \log n$},\\
      \frac{\log^2(n+1)+\log^2n-2\tau^2}{2}+\tau(2\tau-\log(n^2+n)
      &\text{when $\log n<\tau< \log (n+1)$},\\
      -\frac{\log^2(n+1)-\log^2n}{2}+\tau\log\frac{n+1}{n}
      &\text{when $\tau\ge \log(n+1)$}.
    \end{cases}
  \end{equation*}
  The second case is treated by splitting the integral at $u=e^\tau$.
  We compute in this manner that
  \begin{equation*}
    \int_{1}^{10^6}|\Delta(u)|\frac{du}{u}
    \le 0.408.
  \end{equation*}
  We use Lemma~\ref{psitildeasymp} to infer that
  \begin{equation*}
    \int_{10^6}^{10^{19}}|\Delta(u)|\frac{du}{u}
    \le \int_{10^6}^{10^{19}}\frac{1}{u^{3/2}}du\le
    \frac{2}{\sqrt{10^6}}
    =\frac{2}{1000}=0.002.
  \end{equation*}
  We now use Lemma \ref{appLambda3} and Lemma \ref{Ramcorrected} to show that, for some $x_{1}\geq 10^{19}$,
  \begin{align*}
  \int_{10^{19}}^{x_{1}} |\Delta(u)| \frac{du}{u} &\leq\int_{10^{19}}^{x_{1}} \left( \frac{2\cdot 10^{-8}}{u} + \frac{0.05}{u^{3/2}}\right)\, du\\
  &= 2\cdot 10^{-8}(\log x_{1} - 19\log 10) + \frac{0.2}{\sqrt{10^{19}}} - \frac{0.2}{\sqrt{x_{1}}}.
  \end{align*}
To handle the integration beyond $x_{1}$ we use \eqref{lambdalog2} in Lemma \ref{psitildeasymp}, whence the total integral is
$$ 0.408 + 0.002 + 2\cdot 10^{-8}(\log x_{1} - 19\log 10) + \frac{0.2}{\sqrt{10^{19}}} - \frac{0.2}{\sqrt{x_{1}}} + \frac{1.3}{\log x_{1}}.$$
Choosing $x_{1} = \exp(500)$ gives the result.
%  \begin{equation*}
   % \int_{10^{10}}^{U_0}|\Delta(u)|\frac{du}{u}
    %\le 0.0067\int_{10^{10}}^{U_0}\frac{du}{u\log u}
    %=0.0067(\log\log U_0-\log\log 10^{10})\le 0.0166.
 % \end{equation*}
 % Finally, we have
  %\begin{equation*}
 %   \int_{U_0}^{\infty}|\Delta(u)|\frac{du}{u}
   % \le 1.833\int_{U_0}^{\infty}\frac{du}{u\log^2 u}
   % =\frac{1.833}{\log U_0}\le 0.327.
  %\end{equation*}
 % We gather these bounds to infer the result.
\end{proof}
%%%%%%%%%%%%%%%%%%%%%%
We remark that we could further divide the range to use more entries
in the tables in \cite{BKLNW}, but the above result is sufficient for
our purposes.

%%%%%%%%%%%%%%%%%%%%%%%%%%%%%%%%%%%%%%%%%%%%%%%%%%%
\section{Character sum estimates}\label{captain33}
%%%%%%%%%%%%%%%%%%%%%%%%%%%%%%%%%%%%%%%%%%%%%%%%%%%
The work of Stephens and Pintz relied on
the Burgess bound from~\cite{Burgess*62}. Explicit versions of this are
known but are still numerically rather weak. When the modulus is
prime, such bounds have been provided by Francis \cite{Francis*21} improving on work by Trevi\~{n}o
\cite{Trevino*15-2} and McGown \cite{McGown*12}. If we
restrict our attention here to quadratic characters to prime modulus
congruent to $1$ modulo~$4$, we may rely on the slightly stronger
bounds of Booker in \cite{Booker*06}.
Recently, Jain-Sharma, Khale and Liu have produced in
\cite{Jain-Sharma-Khale-Liu*21} an explicit version of the Burgess
inequality for a composite modulus, but only for $q\le\exp\exp(9.6)$.

Instead of the Burgess bound we shall rely on versions of the P\'{o}lya--Vinogradov
inequality. We first require an explicit version of the P\'{o}lya--Vinogradov
inequality due to Frolenkov and Soudararajan in~\cite[Corollary
1]{FS}. In both lemmas that follow, we let $V$ denote the bound on the
character sum. We shall, depending on the conditions, invoke these
bounds for $V$ later in the paper.
%%%%%%%%%%%%%%%%%%%%%%%%
\begin{lem}
  \label{PV}
  When $q\ge100$ and $\chi$ is a non-principal Dirichlet character modulo~$q$, we have
  \begin{equation*}
    \biggl|\sum_{A\le n\le B}\chi(n)\biggr|
    \le\frac{1}{\pi\sqrt{2}}\sqrt{q}(\log q+6)+\sqrt{q}=V.
  \end{equation*}
\end{lem}
%%%%%%%%%%%%%%%%%%%%%%%%
The following is from \cite{Lapkova*16,Lapkova} by Lapkova, which makes a small improvement
on the earlier result from \cite[Theorem 2]{FS} by Frolenkov and Soundararajan.
%%%%%%%%%%%%%%%%%%%%%%%%
\begin{lem}
  \label{PVbis}
  When $q>1$ and $\chi$ is a primitive Dirichlet character modulo $q$, we have
  \begin{equation*}
    \biggl|\sum_{A\le n\le B}\chi(n)\biggr|
    \le
    \begin{cases}
      \frac{2}{\pi^2}\sqrt{q}\log q+0.9467\sqrt{q} + 1.668=V&\text{when $\chi(-1)=1$},\\
      \frac{1}{2\pi}\sqrt{q}\log q+0.8204\sqrt{q} + 1.0286=V&\text{when $\chi(-1)=-1$}.
    \end{cases}
  \end{equation*}
  When $A=0$ and $\chi$ is even, we may divide this bound by~2.
\end{lem}
%%%%%%%%%%%%%%%%%%%%%%%%
%Here is a consequence of the
%result of Trevi\~{n}o (case $r=2$).
%%%%%%%%%%%%%
%\begin{lem}
  %\label{Booker}
 % Let $p\ge 10^{7}$ be a prime number and $\chi$ be a non-principal character modulo~$q$. Let $M$ and $N$ be real numbers with $N\in[1,2p^{5/8}]$. We have
%  \begin{equation*}
  %  \biggl|\sum_{M< n\le M+N}\chi(n)\biggr|
   % \le
   % 3.75\sqrt{N}p^{3/16}(\log p)^{1/4}.
 % \end{equation*}
%\end{lem}
%%%%%%%%%%%%%
Here is a smoothed version of the P\'{o}lya--Vinogradov
that we take from Levin, Pomerance and Soundararajan in
\cite{Levin-Pomerance-Soundararajan*10}.
%%%%%%%%%%%%%%%
\begin{lem}
  \label{LPS}
  Let $\chi$ be a primitive Dirichlet
  character modulo~$q>1$. Let $M$ and $N$ be real numbers with
  $0<N\le q$. With $H(t)=\max(0,1-|t-1|)$, we have
  \begin{equation*}
    \biggl|\sum_{M\le n\le M+2N}\chi(n)
    H\biggl(\frac{n-M}{N}\biggr)
    \biggr|
    \le \sqrt{q}-\frac{N}{\sqrt{q}}.
  \end{equation*}
\end{lem}
%%%%%%%%%%%%%%%

%%%%%%%%%%%%%%%%
\begin{lem}
  \label{Yep}
  Let $\chi$ be a primitive Dirichlet
  character modulo~$q>1$. Let $M$ and $N$ be real numbers with
  $0<N\le q$. When $\chi$ is odd,we have
  \begin{equation*}
    \biggl|\sum_{M< n\le M+N}\chi(n)
    \biggr|
    \le \sqrt{2N}q^{1/4}+\sqrt{q}.
  \end{equation*}
  When $\chi$ is even, we have
  \begin{equation*}
    \biggl|\sum_{n\le N}\chi(n)
    \biggr|
    \le \sqrt{N}q^{1/4}+\tfrac12\sqrt{q}.
  \end{equation*}
\end{lem}
%%%%%%%%%%%%%%%%

%%%%%%%%%%%%%
\begin{proof}
  We may assume that $M$ is an integer. Notice first that the lemma is
  trivial when $N\le \sqrt{q}$, so we may assume $N>\sqrt{q}$.
  Let $K\ge1$ be an integer and
  let $A=N/K$. Keeping the
  notation of Lemma~\ref{LPS}, we first notice that
  \begin{multline*}
    H\biggl(\frac{t-(M-A/2)}{A}\biggr)
    +H\biggl(\frac{t-(M+A/2)}{A}\biggr)
    +\ldots+
    H\biggl(\frac{t-(M+(K-1/2)A)}{A}\biggr)
    \\=
    \begin{cases}
      H\biggl(\frac{t-(M-A/2)}{A}\biggr)
      &\text{when $M-{A}/{2}\le t\le M+A/2$},\\
      1&\text{when $M+A/2\le t\le M+(K+1/2)A$},\\
      H\biggl(\frac{t-(M+(K-1/2)A)}{A}\biggr)
      &\text{when $M+(K-1/2)A\le t\le M+(K+1/2)A$}.
    \end{cases}
  \end{multline*}
  Therefore
  \begin{equation*}
    \biggl|\sum_{M< n\le M+N}\chi(n)
    -\sum_{1\le k\le K}\sum_{n}\chi(n)H\biggl(\frac{t-(M+(k-1/2)A)}{A}\biggr)
    \biggr|
    \le
    \frac{4}{A}\sum_{1\le a\le A/2}a
  \end{equation*}
  which is readily seen to be of size at most
  $\frac{A}{2}+1$. On using Lemma~\ref{LPS}, we get
  \begin{equation}
    \label{bound}
    \biggl|\sum_{M< n\le M+N}\chi(n)
    \biggr|
    \le
    K\sqrt{q}-\frac{KA}{\sqrt{q}}
    +\frac{A}{2}+1
    \le
    K\sqrt{q}-\frac{N}{\sqrt{q}}
    +\frac{N}{2K}+1.
  \end{equation}
  We let $K=1+[q^{-1/4}\sqrt{N/2}]$ and write
  $K = c + q^{-1/4}\sqrt{N/2}$ with $c\in(0,1]$.
  We find that
  \begin{equation*}
    K\sqrt{q}
    +\frac{N}{2K}
    =
    \sqrt{\frac{N\sqrt{q}}{2}}
    +c\sqrt{q}
    + \frac{N}{2c + q^{-1/4}\sqrt{2N}}.
  \end{equation*}
  By computing the derivative with respect to $c$, we check that this
  quantity is maximised at $c=1$. The lemma follows readily.
\end{proof}
%%%%%%%%%%%%%
%%%%%%%%%%%%%%%%%%%%%%%%
\begin{lem}
  \label{L1chi}
  We have
  $
    L(1,\chi)=F(1)\log H+\Ocal^*(VH^{-1})$,
  where $V$ is defined in Lemma~\ref{PV}.
\end{lem}
%%%%%%%%%%%%%%%%%%%%%%%%

%%%%%%%%%%%%%%%%%%%%%%%%
\begin{proof}
  By summation by parts, we find that
  \begin{align*}
    \sum_{n>H}\frac{\chi(n)}{n}
    &=
      \int_H^\infty \sum_{H\le n\le t}\chi(n)\frac{dt}{t^2},
  \end{align*}
  hence
  \begin{align*}
    L(1,\chi)
    &=
      F(1)\log H+f(1)+\int_H^\infty \sum_{H\le n\le
      t}\chi(n)\frac{dt}{t^2}
    \\&=
        F(1)\log H+\sum_{n\le H}\int_H^\infty \frac{dt}{t^2}+\int_H^\infty \sum_{H\le n\le
      t}\chi(n)\frac{dt}{t^2}
    \\&=
    F(1)\log H+\int_H^\infty \sum_{ n\le
      t}\chi(n)\frac{dt}{t^2}=F(1)\log H+\Ocal^*(V/H).\qedhere
  \end{align*}
\end{proof}
%%%%%%%%%%%%%%%%%%%%%%%%

%%%%%%%%%%%%%%%%%%%%%%%%%%%%%%%%%%%%%%%%%%%%%%%%%%% 
\section{Preliminaries to Stephens' approach}\label{captain3}
%%%%%%%%%%%%%%%%%%%%%%%%%%%%%%%%%%%%%%%%%%%%%%%%%%%

%%%%%%%%%%%%%%%%%%%%%%%
From \eqref{deff} and \eqref{defell} in \S \ref{cello} it follows that

  \begin{equation}\label{viola}
    \ell(x)=\sum_{m\le H^x}\frac{\chi(m)\Lambda(m)}{m}f\biggl(x-\frac{\log
      m}{\log H}\biggr).
  \end{equation}

We now recast this for greater ease of use in what follows.

%%%%%%%%%%%%%%%%%%%%%%%
\begin{lem}
  \label{L2}
  We have
  \begin{equation*}
    \frac{\ell(x)}{\log H}=xf(x)-\int_0^x f(u)H^udu/H^x.
  \end{equation*}
  If $H\geq V\geq 1$ we also have $\int_0^1 f(u)H^udu/H=\Ocal^*(1/\log H)$ and
  \begin{equation}\label{pencil}
    \sum_{m\le H}\frac{\chi(m)\Lambda(m)}{m\log H}f(x_m)
    %= f(1)+\Ocal^*\biggl(\frac{(1+\log (H/V))VH^{-1}}{\log H}\biggr).
    = f(1)+\Ocal^*\biggl(\frac{R_\chi(H,V,q)}{\log H}\biggr).
  \end{equation}
  where
  \begin{equation}
    \label{defRHVq}
    R_\chi(H,V,q)=
    \begin{cases}
      (3.66+\log(V^2/q))\frac{\sqrt{q}}{H}
      +\log (4e^2\sqrt{q}H/V^2)\frac{V}{2H}
      &\text{when $\chi$ is even},\\[1em]
      (7.2+\log(V^2/q))\frac{\sqrt{q}}{H}
      +\log (2e^2\sqrt{q}H/V^2)\frac{V}{H}
      &\text{when $\chi$ is odd}.
    \end{cases}
  \end{equation}
\end{lem}
%%%%%%%%%%%%%%%%%%%%%%%
The final proof uses only the upper bound part
of~\eqref{pencil}, see~\eqref{defS}. 
%%%%%%%%%%%%%%%%%%%%%%%
\begin{proof}
  We find that
  \begin{align*}
    \sum_{k\le H^x}\chi(k)\log k
    &=
     \sum_{k\le H^x}\chi(k)\log (H^x)
      -\sum_{k\le H^x}\chi(k)\int_k^{H^x}\frac{dt}{t}
    \\&=
        H^xxf(x)\log H
        -\int_1^{H^x}f\Bigl(\frac{\log t}{\log H}\Bigr)t\frac{dt}{t}
    \\&=
        H^xxf(x)\log H
        -\int_0^{x}f(u)(\log H)H^udu
  \end{align*}
  and the first part of the lemma follows readily.
  Concerning  the upper bound for $|\int_0^1 f(u)H^udu|/H$, we proceed
  as follows.
  %%%%%%%%%%%%%%%%
  \subsubsection*{Case of even characters}
  %%%%%%%%%%%%%%%%
By Lemma~\ref{PVbis} and \ref{Yep}, we have three upper bounds for $|f(u)|$: either
$1$, $q^{1/4}H^{-u/2}+\frac12q^{1/2}H^{-u}$ or $V/(2H^u)$.
We have $q^{1/4}H^{-u/2}+\frac12q^{1/2}H^{-u}\le 1$ when
$H^u/\sqrt{q}\ge 1+\sqrt{3}$. We momentarily set $V^*=V/2$. We define
\begin{equation}
  \label{eq:7}
  u_0= \frac{\log(1+\sqrt{3})+\frac12\log q}{\log H}.
\end{equation}
Define the real parameter $a$ by  $\frac12(1-a)\log H=\log (\sqrt{\sqrt{q}H}/V^*)$. We get
\begin{align*}
 \int_0^1 |f(u)|H^udu
  &\le
    \int_0^{u_0} H^udu+\int_{u_0}^a (q^{1/4}H^{u/2}+\tfrac12q^{1/2})du
    +\int_a^1V^*du
  \\&\le
  \frac{H^{u_0}-1}{\log
  H}+\frac{2q^{1/4}(H^{a/2}-H^{u_0/2})}{\log H}
  +\frac{a-u_0}{2}q^{1/2}+(1-a)V^*
  \\&\le
  \frac{\sqrt{q}}{\log H}
  \bigl(1+\sqrt{3}+2\sqrt{1+\sqrt{3}}\bigr)
  +\frac{2q^{1/4}\sqrt{H}}{\log H}\frac{V^*}{\sqrt{\sqrt{q}H}}
  %%%%%%%%%%%%%%%%%%%
  \\&\qquad+\frac{\log \frac{V^{*2}}{q(1+\sqrt{3})}}{\log H}
  \sqrt{q}
  %%%%%%%%%%%%%%%%%%%
  +\frac{\log(\sqrt{q}H/V^{*2})}{\log H}V^*%-\frac{1}{\log H}
  \\&\le
  \frac{\sqrt{q}}{\log H}
  \left(1+\sqrt{3}+2\sqrt{1+\sqrt{3}}
  +\log \frac{V^{*2}}{q(1+\sqrt{3})}\right)\\
  &\qquad
  +\frac{2V^*}{\log H}+\frac{\log(\sqrt{q}H/V^{*2})}{\log H}V^*.
\end{align*} 
  %%%%%%%%%%%%%%%% 
  \subsubsection*{Case of odd characters}
  %%%%%%%%%%%%%%%%
Again by Lemma~\ref{PVbis} and \ref{Yep}, we have three upper bounds for $|f(u)|$: either
$1$, $q^{1/4}\sqrt{2}H^{-u/2}+q^{1/2}H^{-u}$ or $V/H^u$.
We have $q^{1/4}\sqrt{2}H^{-u/2}+q^{1/2}H^{-u}\le 1$ when
$H^u/\sqrt{q}\ge 2+\sqrt{3}$. We define
\begin{equation}
  \label{eq:7}
  u_0= \frac{\log(2+\sqrt{3})+\frac12\log q}{\log H}.
\end{equation}
Define the real parameter $a$ by  $\frac12(1-a)\log H=\log (\sqrt{2\sqrt{q}H}/V)$. We get
\begin{align*}
 \int_0^1 |f(u)|H^udu
  &\le
    \int_0^{u_0} H^udu+\int_{u_0}^a (q^{1/4}\sqrt{2}H^{u/2}+q^{1/2})du
    +\int_a^1Vdu
  \\&\le
  \frac{H^{u_0}-1}{\log
  H}+\frac{2\sqrt{2}q^{1/4}(H^{a/2}-H^{u_0/2})}{\log H}
  +(a-u_0)q^{1/2}+(1-a)V
  \\&\le
  \frac{\sqrt{q}}{\log H}
  \bigl(2+\sqrt{3}+2\sqrt{2}\sqrt{2+\sqrt{3}}\bigr)
  +\frac{2\sqrt{2}q^{1/4}\sqrt{H}}{\log H}\frac{V}{\sqrt{2\sqrt{q}H}}
  %%%%%%%%%%%%%%%%%%%
  \\&\qquad+\frac{\log \frac{V^2}{2q(2+\sqrt{3})}}{\log H}
  \sqrt{q}
  %%%%%%%%%%%%%%%%%%%
  +\frac{\log(2\sqrt{q}H/V^2)}{\log H}V
  \\&\le
  \frac{\sqrt{q}}{\log H}
  \bigl(2+\sqrt{3}+2\sqrt{2}\sqrt{2+\sqrt{3}}
  +\log \frac{V^2}{2q(2+\sqrt{3})}\bigr)\\
  &\qquad
  +\frac{2V}{\log H}+\frac{\log(2\sqrt{q}H/V^2)}{\log H}V.
\end{align*}
  %%%%%%%%%%%%%%%% 
  \subsubsection*{Resuming the proof}
  %%%%%%%%%%%%%%%%
Inequality \eqref{pencil} follows: indeed, by \eqref{viola}, the left-hand
side is $\ell(1)/\log H$ which we compute with the first formula of the
present lemma. We complete the proof by  using the bound above for
$\int_0^1 |f(u)|H^udu$. 
\end{proof}
%%%%%%%%%%%%%%%%%%%%%%%

%%%%%%%%%%%%%%%%%%%%%%%
\begin{lem}
  \label{L3}
  We have
  \begin{multline*}
    \biggl(x-\frac{1}{\log H}\biggr)F(x)
    =\int_0^x F(x-y)dy
    \\+\sum_{m\le H^x}\frac{\chi(m)\Lambda(m)}{m\log H}
      F\biggl(x-\frac{\log
      m}{\log H}\biggr)+\Ocal^*(1/\log^2 H).
  \end{multline*}
 \end{lem}
%%%%%%%%%%%%%%%%%%%%%%%

%%%%%%%%%%%%%%%%%%%%%%%
\begin{proof}
  On joining \eqref{viola} and  Lemma~\ref{L2}, we get
\begin{equation}
  \label{opt1}
  xf(x)-\int_0^x f(u)H^udu/H^x
  =
  \sum_{m\le H^x}\frac{\chi(m)\Lambda(m)}{m\log H}f\biggl(x-\frac{\log
      m}{\log H}\biggr).
\end{equation}
This is the equivalent of \cite[(55)]{Stephens} by Stephens.
% Specializing
% at $x=1$, and recalling the last estimate of Lemma~\ref{L2}, we get
% the last estimate of our lemma.
The next step is to integrate the above relation:
\begin{align*}
  \int_0^x yf(y)dy-\int_0^x\int_0^y f(u)H^{u-y}dudy
  &=\int_0^x\sum_{m\le H^y}\frac{\chi(m)\Lambda(m)}{m\log H}f\biggl(y-\frac{\log
      m}{\log H}\biggr)dy
  \\&=\sum_{m\le H^x}\frac{\chi(m)\Lambda(m)}{m\log H}
      \int_{\frac{\log m}{\log H}}^x f\biggl(y-\frac{\log
      m}{\log H}\biggr)dy
  \\&=\sum_{m\le H^x}\frac{\chi(m)\Lambda(m)}{m\log H}
      F\biggl(x-\frac{\log
      m}{\log H}\biggr).
\end{align*}
As for the left-hand side, we 
first check that
\begin{equation}
  \label{eq:1}
  \int_0^x yf(y)dy = xF(x)-\int_0^x F(x-y)dy.
\end{equation}
And finally
\begin{align*}
  \int_0^x\int_0^y f(u)H^{u-y}dudy
  &=
    \int_0^xf(u)H^{u}\int_u^x H^{-y}dydu
  \\&=
    \int_0^xf(u)H^{u}\frac{H^{-u}-H^{-x}}{\log H}du
      =\frac{F(x)}{\log H}+\Ocal^*(1/\log^2H)
\end{align*}
by bounding $|f(u)|$ by~1.
\end{proof}
%%%%%%%%%%%%%%%%%%%%%%%

%%%%%%%%%%%%%%%%%%%%%%%
\begin{lem}
  \label{L6}
  We have, when $x\ge0$
  \begin{equation*}
    \int_0^x F(x-y)dy=
    \sum_{m\le H^x}\frac{\Lambda(m)}{m\log H}
      F\biggl(x-\frac{\log
      m}{\log H}\biggr)-F(x)\frac{\gamma}{\log
      H}+\Ocal^*\biggl(\frac{0.411}{\log^2 H}\biggr).
  \end{equation*}
\end{lem}
%%%%%%%%%%%%%%%%%%%%%%%

%%%%%%%%%%%%%%%%%%%%%%%
\begin{proof}
  We start from the right-hand side:
  \begin{align*}
    \sum_{m\le H^x}\frac{\Lambda(m)}{m\log H}
      F\biggl(x-\frac{\log
      m}{\log H}\biggr)
    &=
      \sum_{m\le H^x}\frac{\Lambda(m)}{m\log H}
      \int_0^{x-\frac{\log m}{\log H}}f(t)dt
    \\&=
        \int_0^{x}f(t)
        \sum_{m\le H^{x-t}}\frac{\Lambda(m)}{m\log H}
        dt.
  \end{align*}
  We approximate $\tilde\psi(H^{x-t})$ by
  $(x-t)\log H-\gamma$, getting the main term
  and this is $\int_0^xF(t)dt-F(x)\frac{\gamma}{\log H}$ and treat the
  error term by bounding $|f(t)|$ by 1:
  \begin{align*}
    \int_0^{x}\biggl|f(t)
        \biggl(\frac{\tilde\psi(H^{x-t})}{\log
    H}-x-t-\frac{\gamma}{\log H}\biggr)\biggr|dt
    &\le
      \int_0^{x}\biggl|
      \frac{\tilde\psi(H^{x-t})}{\log
      H}-x-t-\frac{\gamma}{\log H}\biggr|dt
    \\&
         \le
      \int_0^{x}\biggl|
      \frac{\tilde\psi(H^{t})}{\log
      H}-t-\frac{\gamma}{\log H}\biggr|dt 
    \\&
         \le
      \int_1^{H^x}|
      \tilde\psi(u)-\log u-\gamma|\frac{du}{u\log^2H}.
  \end{align*}
  We then majorize this last term by Lemma~\ref{appLambdatilde1}: it
    is not more than $0.411/\log^2H$.
\end{proof}
%%%%%%%%%%%%%%%%%%%%%%%

%%%%%%%%%%%%%%%%%%%%%%%
\begin{lem}
  \label{L7}
  We have, when $x\ge0$,
  \begin{equation*}
    xF(x)=
    \sum_{m\le H^x}\frac{\Lambda(m)(1+\chi(m))}{m\log H}
      F\biggl(x-\frac{\log
      m}{\log H}\biggr)
   %%%%%%%%%%%%%%%%%%%%%%%%%%%%%%%
    +F(x)\frac{1-\gamma}{\log
      H}+\Ocal^*\biggl(\frac{1.411}{\log^2 H}\biggr).
  \end{equation*}
\end{lem}
%%%%%%%%%%%%%%%%%%%%%%%

%%%%%%%%%%%%%%%%%%%
\begin{proof}
  Join the first equality of Lemma~\ref{L3} together with Lemma~\ref{L6}.
\end{proof}
%%%%%%%%%%%%%%%%%%%
%This lemma is the equivalent of \autocite[Lemma 1]{Stephens}.
%%%%%%%%%%%%%%%%%%%%%%%%%%%%%%%%%%%%%%%
\section{A comparison and the main inequality}\label{captain4}
%%%%%%%%%%%%%%%%%%%%%%%%%%%%%%%%%%%%%%%

This section is devoted to the comparison between
\begin{equation*}
  \sum_{m\le H}\frac{\Lambda(m)}{m\log H}
  f(x_m)
\end{equation*}
and $F(1)$. The important observation, essentially due to Stephens, is
that since $f$ has tame variations, both should be about
equal.
One look at the final proof discloses that it is enough to bound the
initial sum \emph{from below} by $F(1)$ plus some error term.

%%%%%%%%%%%%%%%%%%
%%%%%%%%%%%%%%%%%%
We first connect $F(1)$ with the bounds on character sums, that is, with the $V$ from Lemmas \ref{PV} and \ref{PVbis}.
%%%%%%%%%%%%%%%%%%%%%%%%%

%%%%%%%%%%%%%%%%%%%%%%%%%%%%%%%%
\begin{lem}
  \label{F1}
  We have, for any $D\ge1$,
  \begin{equation*}
    F(1)=\sum_{n\le H/D}\frac{\chi(n)}{n\log H}
    -\frac{f(1)}{\log H}+\Ocal^*\biggl(\frac{(D-1)V/H}{\log H}\biggr).
  \end{equation*}
\end{lem}
%%%%%%%%%%%%%%%%%%%%%%%%%%%%%%%%

%%%%%%%%%%%%%%%%%%%%%%%%
\begin{proof}
    We have
  \begin{align*}
      \sum_{H/D<n\le H}\frac{\chi(n)}{n}
      &=
        \sum_{H/D<n\le H}\chi(n)\biggl(\int_n^H\frac{dt}{t^2}+\frac{1}{H}\biggr)
      \\&=
        \int_{H/D}^H\sum_{H/D<n\le t}\chi(n)\frac{dt}{t^2}+\frac{1}{H}\sum_{H/D<n\le H}\chi(n).
    \end{align*}
    Using that $|\sum_{A\le n \le B} \chi(n)| \le V$ yields the desired result.
\end{proof}
%%%%%%%%%%%%%%%%%%%%%%%%
%%%%%%%%%%%%%%%%%%%%%
\begin{lem}
  \label{complicated}
    Let $D_0\geq 1$, and $A_2$ be such that 
    \begin{equation}\label{a2eq}
        |\psi(x)-x|\leq A_2\sqrt{x},\quad D_0\leq x\leq 10^{19}.
    \end{equation}
    We then have, for any $D\geq D_0$,
  \begin{equation*}
    \sum_{mn\le H}\chi(n)\Lambda(m)
    =H\bigl(F(x_D)\log H+f(x_D)\bigr)
    +\Ocal^*\bigl(1.6\cdot 10^{-5} H+2A_2HD^{-1/2}+1.04DV\bigr).
  \end{equation*}
\end{lem}
%%%%%%%%%%%%%%%%%%%%%
Please notice that we would need only the lower estimate in the last bound.
%%%%%%%%%%%%%%%%%%%%%
\begin{proof}
    We write
    \begin{align}\label{chair}
        \sum_{mn\leq H}\chi(n)\Lambda(m)&=\sum_{n\leq H/2}\chi(n)\sum_{m\leq H/n}\Lambda(m)\notag\\
        &=\sum_{n\leq H/D}\chi(n)\sum_{m\leq H/n}\chi(m)+\sum_{m\leq D}\Lambda(m)\sum_{H/D<n\leq H/m}\chi(n).
    \end{align}
    By Lemma 3, the last sum over $n$ is bounded in absolute value by
    $V$. It then follows by Lemma \ref{appLambda0} that the second
    summand of \eqref{chair} satisfies
    \begin{equation*}
        \sum_{m\leq D}\Lambda(m)\sum_{H/D<n\leq H/m}\chi(n)\leq 1.04DV.
    \end{equation*}
    Concerning the first summand of \eqref{chair}, we use three steps.
    For the first step, we restrict to the range $H/10^{19}<n\leq H/D$
    and use \eqref{a2eq}. Note that Lemma \ref{appLambda2} tells us
    that we can take $A_2=0.94$ provided $D_0>11$. A quick calculation
    also shows that $A_2=\sqrt{2}$ works for $D_0\geq 1$, or
    $A_2=0.956$ works for $D_0\geq 7$. Now,
    \begin{align*}
        &\sum_{\frac{H}{10^{19}}<n\leq\frac{H}{D}}\chi(n)\sum_{m\leq H/n}\Lambda(m)\\
        &\qquad=\sum_{\frac{H}{10^{19}}<n\leq\frac{H}{D}}\chi(n)\sum_{m\leq H/n}\Lambda(m)\\
        &\qquad=H\sum_{\frac{H}{10^{19}}<n\leq\frac{H}{D}}\frac{\chi(n)}{n}
          +\mathcal{O}^*\left(\sum_{\frac{H}{10^{19}}<n\leq\frac{H}{D}}
          A_2\sqrt{\frac{H}{n}}\right)\\
        &\qquad=H\sum_{\frac{H}{10^{19}}<n\leq\frac{H}{D}}\frac{\chi(n)}{n}
          +\mathcal{O}^*\left(2A_2HD^{-1/2}\right),
    \end{align*}
    where for the second equality we used Lemma \ref{loc1}. 
    
    For the second step, we use Lemma \ref{appLambda3} and consider the range $H/A<n\leq H/10^{19}$, where $A\geq\exp(40)$ is to be chosen later. That is,
    \begin{align}\label{chairs}
        \sum_{\frac{H}{A}<n\leq\frac{H}{10^{19}}}\chi(n)\sum_{m\leq H/n}\Lambda(m)&=H\sum_{\frac{H}{A}<n\leq\frac{H}{10^{19}}}\frac{\chi(n)}{n}+\mathcal{O}^*\left(1.93378\cdot 10^{-8}\sum_{\frac{H}{A}<n\leq\frac{H}{10^{19}}}\frac{H}{n}\right)\notag\\
        &=H\sum_{\frac{H}{A}<n\leq\frac{H}{10^{19}}}\frac{\chi(n)}{n}+\mathcal{O}^*\left(1.93378\cdot 10^{-8}(1+\log(A/10^{19}))H\right),
    \end{align}
    where for the second equality we used Lemma \ref{loc2}.
    
    For the third step we consider the sum over $n\leq H/A$. First, if $H<A$ then there is nothing to add. On the other hand, if $H\geq A$ we use Lemma \ref{appLambda} to get 
    \begin{equation}\label{table}
        \sum_{n\leq\frac{H}{A}}\chi(n)\sum_{m\leq H/n}\Lambda(m)=H\sum_{n\leq\frac{H}{A}}\frac{\chi(n)}{n}+\mathcal{O}^*\left(\sum_{n\leq\frac{H}{A}}\frac{1.83H}{n\log^2(H/n)}\right).
    \end{equation}
    Since $n\log^2 (H/n)$ is increasing when $n\leq H/e^2$, we have that
    \begin{align*}
        \sum_{n\leq H/A}\frac{1.83H}{n\log^2(H/n)}&\leq\frac{1.83H}{\log^2H}+H\int_{1}^{H/A}\frac{1.83dt}{t\log^2(H/t)}\\
        &\leq\frac{1.83H}{\log^2H}+1.83H\left(\frac{1}{\log A}-\frac{1}{\log H}\right)\\
        &=1.83H\left(\frac{1}{\log^2H}+\frac{1}{\log A}-\frac{1}{\log H}\right).
    \end{align*}
    Since the above is decreasing in $H$, and $H\geq A$, we can set $A=\exp(574)$ to bound the $\mathcal{O}^*$ terms in \eqref{chairs} and \eqref{table} by $1.6\cdot 10^{-5}H$.
\end{proof}

%%%%%%%%%%%%%%%%%%%%%

%%%%%%%%%%%%%%%%%%%%%%
\begin{lem}\label{FirstApp}
    For any $D\geq D_{0}\geq 1$
    \begin{equation*}
        \left|\sum_{m\leq H}\frac{\Lambda(m)}{m\log H}f(x_m)-F(1)-\frac{f(1)}{\log H}\right|\leq\frac{1.6\cdot 10^{-5}+2A_2D^{-1/2}+(2.04D-1)VH^{-1}}{\log H},
    \end{equation*}
    where $A_{2}$ is as in Lemma \ref{complicated}.
\end{lem}
%%%%%%%%%%%%%%%%%%%%%%
The main proof only requires a lower bound
for $\sum_{m\leq H}{\Lambda(m)f(x_m)}/m$, see~\eqref{whistle}.
%%%%%%%%
\begin{proof}
  By using the definition of $f$, we get
  \begin{equation*}
    \sum_{m\le H}\frac{\Lambda(m)}{m\log H}
    f(x_m)
    =
      \sum_{mn\le H}\frac{\chi(n)\Lambda(m)}{H\log H}
    \end{equation*}
    and we appeal to Lemma~\ref{complicated}.
  This leads to
  \begin{equation*}
    \sum_{m\le H}\frac{\Lambda(m)}{m\log H}
    f(x_m)
    =
    \sum_{n\le H/D}\frac{\chi(n)}{n\log H}+\Ocal^*\biggl(
    \frac{1.6\cdot 10^{-5}+2A_2D^{-1/2}+1.04DVH^{-1}}{\log H}
    \biggr).
  \end{equation*}
  Note further that, by Lemma~\ref{F1} (we need only the upper estimate), we have
  \begin{equation}\label{post}
    F(1)=\sum_{n\le H/D}\frac{\chi(n)}{n\log H}
    -\frac{f(1)}{\log H}+\Ocal^*\biggl(\frac{(D-1)V/H}{\log H}\biggr).
  \end{equation}
\end{proof}
%%%%%%%%

%%%%%%%%%%%%%%%%%%

%%%%%%%%%%%%%%%%%%%%%%%%%%%%%%%%%%%%%%% 
%\section{The main inequality}
%%%%%%%%%%%%%%%%%%%%%%%%%%%%%%%%%%%%%%%

We are now in a position to prove the following crucial lemma.

%%%%%%%%%%%%%%%%%%%%%%%
\begin{lem}
  \label{L4}
   Let $H\geq 10^6$ and $x\geq 1/2$. Then we have
  \begin{equation*}
    0\le \frac{h(1,x)+ h(\chi,x)}{\log H}\le 2x
  \end{equation*}
  as well as, if $H$ also satisfies $H\geq V$,
  \begin{equation*}
  \begin{split}
    \frac{h(1,x)+ h(\chi,x)}{\log H}\le & 2-F(1)
    + f(1)-\frac{f(1)}{\log H}\\
    &+\frac{-1.15+3.81A_2^{2/3}(V/H)^{1/3}-VH^{-1}+R_\chi(H,V,q)}{\log H},
    \end{split}
  \end{equation*}
  where $A_2$ is as in Lemma \ref{complicated} when $D$ is taken to be
  $\left(\frac{A_2}{2.04}\frac{H}{V}\right)^{2/3}$ (see
  \eqref{macbeth}).
\end{lem}
%%%%%%%%%%%%%%%%%%%%%%%
This is the equivalent of \cite[Lemma 2]{Stephens} by Stephens.

%%%%%%%%%%%%%%%%%%%%%%%
\begin{proof}
  The first inequality follows by
  Lemma~\ref{appLambdatilde}. Concerning the second one, we proceed as
  follows. Define  
  \begin{equation}
    \label{defS}
    S=\sum_{m\le H}\frac{\Lambda(m)}{m\log H}
    \bigl(1-f(x_m)\bigr)
    -\sum_{m\le H}\frac{\chi(m)\Lambda(m)}{m\log H}
    \bigl(1-f(x_m)\bigr).
  \end{equation}
  Since $|\chi(m)|$, $|f(x_m)|\leq 1$,  we have $S\ge0$. Furthermore, on expanding and using
  the second part of Lemma~\ref{L2}, we find that
  \begin{multline}
    \label{eq:2}
    S=\sum_{m\le H}\frac{\Lambda(m)}{m\log H}
    -\sum_{m\le H}\frac{\Lambda(m)}{m\log H}
    f(x_m)
    \\-\frac{h(\chi,1)}{\log H}+  f(1)
    +\Ocal^*\biggl(\frac{R_\chi(H,V,q)}{\log H}\biggr).
  \end{multline}
  Now we use Lemma~\ref{FirstApp}.  Since $S\ge0$, this leads to the
  inequality
  \begin{equation}\label{whistle}
  \begin{split}
    \frac{h(\chi,1)}{\log H}\le & \frac{h(1,1)}{\log
      H}-F(1)-\frac{f(1)}{\log H}
     + f(1)\\
     &+\frac{1.6\cdot10^{-5}+2A_2D^{-1/2}+(2.04D-1)VH^{-1}+R_\chi(H,V,q)}{\log H}.
     \end{split}
  \end{equation}
  We select
  \begin{equation}\label{macbeth}
    D=\Bigl(\frac{A_2}{2.04}\frac{H}{V}\Bigr)^{2/3}
  \end{equation}
so that the expression in (\eqref{whistle}) involving $D$ is minimised. This then gives
 % and note that $2A_2(1.04/\sqrt{2})^{1/3}+1.04(A_2/1.04)^{2/3}\le 3.83$ to get
  \begin{equation*}
  \begin{split}
    \frac{h(\chi,1)}{\log H}\le & \frac{h(1,1)}{\log
      H}-F(1)-\frac{f(1)}{\log H}
    + f(1)\\
    &+\frac{1.6\cdot10^{-5}+3.81A_2^{2/3}(V/H)^{1/3}-VH^{-1}+R_\chi(H,V,q)}{\log H}
    \end{split}.
  \end{equation*}
  Let us extend this inequality to $h(\chi,x)$. We simply write
  \begin{align*}
    h(\chi,x)
    &=h(\chi,1)-\sum_{H^x<m\le H}\frac{\chi(m)\Lambda(m)}{m}
    \\&\le h(\chi,1)+h(1,1)-h(1,x),
  \end{align*}
  hence the result, since $2h(1,1)/\log H\le 2-2\times0.576/\log H$ by
  Lemma~\ref{appLambdatilde} and  $1.6\cdot10^{-5} -2\times0.576\le -1.15$.
\end{proof}
%%%%%%%%%%%%%%%%%%%%%%% 

%%%%%%%%%%%%%%%%%%%%%%%
%%%%%%%%%%%%%%%%%%%%%%% 

%%%%%%%%%%%%%%%%%%%%%%%%%%%%%%%%%%%%%%%
\section{A result in optimization}\label{captain5}
%%%%%%%%%%%%%%%%%%%%%%%%%%%%%%%%%%%%%%%

This section contains a refined version of a theorem of Stephens. No further arithmetical material
is being introduced. We start with a
technical lemma.
%%%%%%%%%%%%%%%%%%%%%%%%%%%%%%%%%%
\begin{lem}
  \label{aux}
  We have
  \begin{multline*}
    -4\int_\theta^x(x-u)\log u\, du
  +
  2\int_{x-\theta}^\theta udu+\int_\theta^x 2\theta du
  \\=2x(x-x\log x-\theta)
  +(2x-\theta)\theta(1+2\log\theta).
  \end{multline*}
\end{lem}
%%%%%%%%%%%%%%%%%%%%%%%%%%%%%%%%%%
%%%%%%%%%%%%%%%%%%%%%%%%%%%%%%%%%%
\begin{proof}
  Notice that $2\int u\log u du=u^2\log u-(u^2/2)$ and thus
  \begin{align*}
    4\int_\theta^x(x-u)\log u\, du
    &=4x(x\log x-x-\theta\log\theta+\theta)-2x^2\log x+x^2+2\theta^2\log \theta-\theta^2
    \\&=4x(-\theta\log\theta+\theta)+2x^2\log x-3x^2+2\theta^2\log \theta-\theta^2.
  \end{align*}
  Next,
  \begin{equation*}
  2\int_{x-\theta}^\theta udu+\int_\theta^x 2\theta du
  =\theta^2-(x-\theta)^2+2\theta(x-\theta)=-x^2+4x\theta-2\theta^2,
\end{equation*}
and thus
\begin{multline*}
  -4\int_\theta^x(x-u)\log u\, du
  +
  2\int_{x-\theta}^\theta udu+\int_\theta^x 2\theta du
  \\=
  4x(\theta\log\theta-\theta)-2x^2\log x+3x^2-2\theta^2\log
  \theta+\theta^2
  %%%%%%%%%%%%%
  -x^2+4x\theta-2\theta^2,
 % \\= -2x^2\log x+2x^2  -2\theta^2\log\theta-\theta^2  +4x\theta\log\theta
%  \\= x\varphi(x)  +(2x-\theta)\theta(1+2\log\theta).
\end{multline*}
whence the lemma follows after some simple algebraic rearrangement.
\end{proof}
%%%%%%%%%%%%%%%%%%%%%%%%%%%%%%%%%%

%%%%%%%%%%%%%%%%%%%%%%%%%%%%%%%%%%%%%%%%%%%%%%%%%%%%
\begin{lem}
  \label{opt}
  Let $H>1$ be a real parameter. Suppose we are given a sequence of
  non-negative real numbers $(u_m)_{1\le m\le H}$ and a continuous
  function $G$ over $[0,1]$. Assume we have, for every
  $x\in[0,1]$, that
  \begin{equation}
    \tag{$H_0$}
    G(x)\le x,
  \end{equation}
  that for some parameters $a$ and  $\ve_2$, we have,
  when $x\ge 1/2$,
  \begin{equation}
    \tag{$H_1$}
    (x+a)G(x)
    \le \sum_{m\le H^x}\frac{u_m}{\log H}
    G\biggl(x-\frac{\log m}{\log H}\biggr)
    +\ve_2,
  \end{equation}
  that
  \begin{equation}
    \tag{$H_2$}
    0\le \sum_{m\le H^x}\frac{u_m}{\log H}\le 2x
  \end{equation}
  and that, for some parameter $\ve_1$ we have, when $x\ge 1/2$,
  \begin{equation}
    \tag{$H_3$}
    \sum_{m\le H^x}\frac{u_m}{\log H}\le 2-G(1)+\ve_1.
  \end{equation}
  Then either $G(1)\le 2(1-1/\sqrt{e})$ or
  \begin{equation}
    2a\theta\log\theta
    -
    2\theta(1/\sqrt{e}-\theta)(2+\log \theta)
    +\ve_1+\ve_2\ge0
  \end{equation}
  where $\theta=1-G(1)/2$ belongs to $[1/2,1/\sqrt{e}]$.
\end{lem}
%%%%%%%%%%%%%%%%%%%%%%%%%%%%%%%%%%%%%%%%%%%%%%%%%%%%

\begin{proof}
Set
\begin{equation}
  \label{deftheta}
  \theta=1-G(1)/2,\quad \varphi(y)=2(y-y\log y-\theta).
\end{equation}
The function $\varphi$ is increasing (its derivative is $-2\log y$) on $(0,1]$ and
takes the positive value $-2\theta\log \theta$ at $y=\theta$.
Note that $\theta\ge1/2$ since $G(1)\le 1$, and that when
$\theta\ge1/\sqrt{e}$, our result is immediate. 
Let us assume that $\theta < 1/\sqrt{e}$ so that $\theta+2\theta\log\theta<0$.
Assume that, when $\theta\le y\le Z$, we have $G(y)\le
\varphi(y)$. This latter inequality translates into
\begin{equation*}
  G(1)-G(y)\ge 2(1-y+y\log y).
  %g(t)=1-t+t*log(t)
\end{equation*}
Our initial remark is that $\theta$ is such a number.
\begin{proof}
  Indeed, if it where not, we would have
  \begin{equation*}
    G(1)=G(\theta)+G(1)-G(\theta)\le \theta + 2(1-\theta+\theta\log \theta)
  \end{equation*}
  since $G(x)\le x$. We notice next that $G(1)=2-2\theta$, so that the
  above inequality can be rewritten as $G(1)\le
  G(1)+\theta+2\theta\log\theta<G(1)$ by the inequality assumed for
  $\theta$, leading to a contradiction.
\end{proof}
We define for this proof
\begin{equation}
  \label{defg}
  g(y)=\sum_{m\le H^y}u_m/\log H.
\end{equation}
We find that, for $Z\ge x\ge\theta$,
\begin{multline*}
  \sum_{m\le H^x}\frac{u_m}{\log H}
      G\biggl(x-\frac{\log
      m}{\log H}\biggr)
  \le
  \sum_{m\le H^{x-\theta}}\frac{u_m}{\log H}
  \varphi\biggl(x-\frac{\log
    m}{\log H}\biggr)
  \\
  +
  \sum_{H^{x-\theta}<m\le H^x}\frac{u_m}{\log H}
  \biggl(x-\frac{\log
    m}{\log H}\biggr)
\end{multline*}
by bounding above $G(y)$ by $y$ $(H_0$) when $y\le \theta$.
We study separately the two right-hand side sums, say $S_1$ and $S_2$. First we note
that, on recalling the definition~\eqref{defg} of $g$:
\begin{align*}
  S_1
  &=
    \sum_{m\le H^{x-\theta}}\frac{u_m}{\log H}
    \int_\theta^{x-\frac{\log m}{\log H}}\varphi'(u)du
    +g(x-\theta)\varphi(\theta)
  \\&=
      g(x-\theta)\varphi(\theta)
      -2\int_\theta^x g(x-u)\log u\,du
\end{align*}
while
\begin{equation*}
  S_2
  =
    \sum_{H^{x-\theta<m\le H^x}}\frac{u_m}{\log H}
    \int_{\frac{\log m}{\log H}}^x du
  =
    \int_{x-\theta}^x (g(u)-g(x-\theta))du
\end{equation*}
and this amounts to
\begin{equation*}
  S_1+S_2
  =
    g(x-\theta)(\varphi(\theta)-\theta)
    -2\int_\theta^x g(x-u)\log u\, du
    +\int_{x-\theta}^x g(u)du.
\end{equation*}
In the first integral, we bound above $g(x-u)$ by
$2(x-u)$ by $(H_2)$.
We split the second integral at $u=\theta$; between $x-\theta$ and
$\theta$, we bound above $g(u)$ again by $2u$ while in the
later range, we bound above $g(u)$ by $2\theta+\ve_1$ by
$(H-3)$ (valid since $u\ge \theta\ge 1/2$). 
We infer in this manner that
\begin{multline*}
  S_1+S_2
  \le
    g(x-\theta)(\varphi(\theta)-\theta)
   \\ -4\int_\theta^x(x-u)\log u\, du
    +2\int_{x-\theta}^\theta udu
    +\int_\theta^x (2\theta+\ve_1)du.
\end{multline*}
By Lemma~\ref{aux} and noticing that
$\varphi(\theta)-\theta=-\theta(1+2\log\theta)$, we get (again bounding above $g(x-u)$ by
$2(x-u)$ by $(H_2)$)
\begin{align*}
  S_1+S_2
  &\le
  x\varphi(x)
  +
  \bigl(g(x-\theta)+\theta-2x\bigr)
  (\varphi(\theta)-\theta)
  +\ve_1
  \\&\le
  x\varphi(x)
  +
  \theta^2
  (1+2\log\theta)
  +\ve_1.
\end{align*}
By $(H_1)$ and the above, we infer that
\begin{equation*}
  (x+a)G(x)\le
  x\varphi(x)
  +
  \theta^2
  (1+2\log\theta)
  +\ve_1+\ve_2.
\end{equation*}
We also find that, when $\theta\le 1/\sqrt{e}$, we have
\begin{equation*}
  0-\theta^2
  (1+2\log\theta)
  =\int_\theta^{1/\sqrt{e}}2u(2+\log u)du\notag
  \ge 2\theta(1/\sqrt{e}-\theta)(2+\log \theta).
\end{equation*}
Hence, we get
\begin{equation*}
  (x+a)G(x)\le
  (x+a)\varphi(x)
  -a\varphi(x)
  -
  2\theta(1/\sqrt{e}-\theta)(2+\log \theta)
  +\ve_1+\ve_2.
\end{equation*}
We can now use $\varphi(x)\ge \varphi(\theta)=-\theta\log \theta$,
getting
\begin{equation*}
  (x+a)G(x)\le
  (x+a)\varphi(x)
  +2a\theta\log\theta
  -
  2\theta(1/\sqrt{e}-\theta)(2+\log \theta)
  +\ve_1+\ve_2.
\end{equation*}
When $2a\theta\log\theta
  -
  2\theta(1/\sqrt{e}-\theta)(2+\log \theta)
  +\ve_1+\ve_2<0$, we would have
$G(x)<\varphi(x)$. However the function $G$ is continuous and
$G(1)=\varphi(1)$, there exists an $x_0$ between $\theta$ and $1$ for
which $G(x_0)=\varphi(x_0)$ and $G(x)\le\varphi(x)$ for $x$ between
$\theta$ and $x_0$. The above inequality then leads to a
contradiction.
Hence we have
\begin{equation*}
  2a\theta\log\theta
  -
  2\theta(1/\sqrt{e}-\theta)(2+\log \theta)
  +\ve_1+\ve_2\ge0.\qedhere
\end{equation*}
\end{proof}

%%%%%%%%%%%%%%%%%%%%%%%%%%%%%%%%%%%%%%%
\section{Proof of Theorems \ref{goal} and \ref{goal2}}\label{captain6}
%%%%%%%%%%%%%%%%%%%%%%%%%%%%%%%%%%%%%%%

We use Lemma~\ref{opt} with $G(x)=F(x)$ and
$u_m=(1+\chi(m))\Lambda(m)/m$.
%%%%%%%%%%%%%%%%
\subsubsection*{Initial upper bound}
%%%%%%%%%%%%%%%%
% Lemma~\ref{L1chi0} gives us
% \begin{equation}
%   \label{IniMajo}
%   L(1,\chi)\le F(1)\log H+ C_0+\frac{c_1\sqrt{q}}{H}
% \end{equation}
% where $(C_0,c_1)=(0.47,2/5)$ when $\chi$ is even and
% $(C_0,c_1)=(0.266,4\pi/9)$ when $\chi$ is odd and where we assume that
% $H\ge\sqrt{q}\ge 1000$.
Lemma~\ref{L1chi} gives us
\begin{equation}
  \label{IniMajo}
  L(1,\chi)\le F(1)\log H+ \frac{V}{H}.
\end{equation}

%%%%%%%%%%%%%%%%
\subsubsection*{Hypotheses $(H_0)$, $(H_1)$, $(H_2)$ and $(H_3)$}
%%%%%%%%%%%%%%%%

Hypothesis $(H_0)$ is granted by the
bound $|f(u)|\le 1$. By Lemma \ref{L7} we can then set 
\begin{equation}
  \label{defve2}
  a\le\frac{\gamma-1}{\log H},\quad \ve_2\le \frac{1.411}{\log^2 H},
\end{equation}
and this gives us Hypothesis $(H_1)$.

Lemma~\ref{psitildeasymp} is enough to grant Hypothesis~$(H_2)$.
Finally, by Lemma~\ref{L4} and provided that $H\geq\max(V,10^6)$, Hypothesis
$(H_3)$ is satisfied with
\begin{equation}
  \label{defve1}
  \ve_1\le f(1)-\frac{f(1)}{\log H}
  +\frac{-1.15+3.81A_2^{2/3}(V/H)^{1/3}-VH^{-1}+R_\chi(H,V,q)}{\log H}.
\end{equation}
We will further majorize $f(1)$ by $V/2$ when $\chi$ is even and by
$V/H$ when $\chi$ is odd.

%%%%%%%%%%%%%%%%
\subsubsection*{Using Lemma~\ref{opt}}
%%%%%%%%%%%%%%%%
So we infer that
\begin{equation}
  \label{eq:4}
  %L(1,\chi)\le 2(1-\theta)\log H +C_0+\frac{c_1\sqrt{q}}{H}
  L(1,\chi)\le 2(1-\theta)\log H +\frac{V}{H}
\end{equation}
where $\theta\in[1/2,1/\sqrt{e}]$ satisfies
\begin{equation}\label{thetaeq}
    2a\theta\log\theta
    -
    2\theta(1/\sqrt{e}-\theta)(2+\log \theta)
    +\ve_1+\ve_2\ge0.
  \end{equation}
  Since $a<0$, if this inequality is satisfied for $H_0$ then it
  remains true for $H\ge H_0$. We select $H=BV$, for some
  parameter $B$, and bound $|f(1)|$ by $V/H$. Therefore, for
  $q\geq q_{0}$ we have that
\begin{equation}\label{ivory}
D \geq \left(\frac{A_{2}B}{2.04 } \right)^{2/3}=D_0.
\end{equation}
So, here are possible choices:
\begin{align}
  \label{eq:3}
  &B\ge \sqrt{2}/2.04 \rightarrow A_2=\sqrt{2},\\
  &B\ge 39.6 \rightarrow A_2=0.956,\\
  &B\ge 79.5\rightarrow A_2=0.94.
\end{align}
%%%%%%%%%%%%%%%%
\subsubsection*{Setting the numerics}
%%%%%%%%%%%%%%%%

We can now prove Theorems \ref{goal} and \ref{goal2}.
We use the expression for $V$ given in Lemma
\ref{PV}.
We take $H=BV$ which we assume to be $\ge 10^6$, we also assume that
$q\ge q_0$ so that $V\ge V_0$. Given a choice of $B$, we select
\begin{align}
  \label{eq:5}
  a&=\frac{\gamma-1}{\log( BV_0)},\\
  \ve_2&=\frac{1.411}{\log(BV_0)^2},\\
  \ve_1&=\frac{\delta(\chi)}{B}-\frac{\delta(\chi)}{B\log (BV_0)}
  +\frac{-1.15+3.81A_2^{2/3}B^{-1/3}-B^{-1}+R(BV_0,V_0,q_0)}{\log (BV_0)}.
\end{align}
where $\delta(\chi)=(3-\chi(-1))/4$.
We then compute the smallest solution $\theta^*$ to~\eqref{eq:4} and infer that
\begin{equation}
  \label{eq:6}
  \frac{L(1,\chi)}{\log q}\le 2(1-\theta^*)\frac{\log B+\log V_0}{\log
    q_0}
  % +\frac{C_0}{\log q_0}+\frac{c_1\sqrt{q_0}}{BV_0\log q_0}.
  +\frac{1}{B\log q_0}.
\end{equation}

\textbf{Result for $\chi$ even and primitive:}
  We select $B=51$, and infer that $L(1,\chi)<\frac{1}{2}\log q$ when
  $q\ge 7\cdot10^{22}$. But this is already known for all $q$'s by
  \cite{Ramare*01}. Even more is true if we combine the theorem of
  Saad Eddin in \cite{Saad-Eddin*16} together with \cite[Corollary 1]{Ramare*02c}.

  We select $B=80$, and infer that $L(1,\chi)<\frac{9}{20}\log q$ when
  $q\ge 2\cdot10^{49}$.

  \textbf{Result for $\chi$ odd and primitive:}
  We select $B=90$, and infer that $L(1,\chi)<\frac{1}{2}\log q$ when
  $q\ge 2\cdot 10^{23}$.

  We select $B=145$, and infer that $L(1,\chi)<\frac{9}{20}\log q$ when
  $q\ge 5\cdot10^{50}$.

\subsection*{Acknowledgements}
We are grateful to Enrique Trevi\~{n}o for some preliminary discussions on this topic.

 %\printbibliography
\bibliographystyle{plain}

\bibliography{JRT-v10.bib}

\end{document}